\documentclass[11pt,a4paper]{amsart}
\usepackage{graphicx,multirow,array,amsmath,amssymb}

\newtheorem{theorem}{Theorem}

\newtheorem{lemma}[theorem]{Lemma}

\newtheorem{definition}{Definition}

\begin{document}

\title{Period and toroidal knot mosaics}

\author[S. Oh]{Seungsang Oh}
\address{Department of Mathematics, Korea University, Seoul 02841, Korea}
\email{seungsang@korea.ac.kr}
\author[K. Hong]{Kyungpyo Hong}
\address{National Institute for Mathematical Sciences, Daejeon 34047, Korea}
\email{kphong@nims.re.kr}
\author[H. Lee]{Ho Lee}
\address{Department of Mathematical Sciences, KAIST, Daejeon 34141, Korea}
\email{figure8@kaist.ac.kr}
\author[H. J. Lee]{Hwa Jeong Lee}
\address{School of Undergraduate Studies, DGIST, Daegu 42988, Korea}
\email{hjwith@dgist.ac.kr}
\author[M. J. Yeon]{Mi Jeong Yeon}
\address{Department of Mathematics, Kyung Hee University, Seoul 02447, Korea}
\email{ym-501@daum.net}

\thanks{The corresponding author(Seungsang Oh) was supported by the National Research Foundation of Korea(NRF) grant funded by the Korea government(MSIP) (No. NRF-2014R1A2A1A11050999).}
\thanks{Hwa Jeong Lee was supported by Basic Science Research Program through the National Research Foundation of Korea (NRF) funded by the Ministry of Science, ICT $\&$ Future Planning (NRF-2015R1C1A2A01054607).}

\begin{abstract}
Knot mosaic theory was introduced by Lomonaco and Kauffman in the paper on 
`Quantum knots and mosaics' to give a precise and workable definition of quantum knots, 
intended to represent an actual physical quantum system.
A knot $(m,n)$--mosaic is an $m \! \times \! n$ matrix whose entries are eleven mosaic tiles,
representing a knot or a link by adjoining properly.
In this paper we introduce two variants of knot mosaics:
period knot mosaics and toroidal knot mosaics,
which are common features in physics and mathematics.
We present an algorithm producing the exact enumeration of
period knot $(m,n)$--mosaics for any positive integers $m$ and $n$,
toroidal knot $(m,n)$--mosaics for co-prime integers $m$ and $n$,
and furthermore toroidal knot $(p,p)$--mosaics for a prime number $p$.
We also analyze the asymptotics of the growth rates of their cardinality.
\end{abstract}

\maketitle

\section{Introduction}

One of remarkable discovery in knot theory is the Jones polynomial,
and it is not only of mathematical interest but also an essential ingredient 
to quantum theory \cite{J1, J2, K1, K2, L, LK2, SJ}.
In 2004, Lomonaco and Kauffman introduced knot mosaic system to give 
a definition of quantum knot system \cite{LK1, LK3, LK4, LK5}.
This definition is intended to represent an actual physical quantum system.
These quantum knots are superpositions of knots 
whose projections can be constructed as grids consisting of suitably connected tiles.
It is known that all tame knots can be constructed in this way.

In this paper we introduce two variants of a knot mosaic,
which are common features in physics and mathematics.
One is a period knot mosaic in the whole plane whose periodic patches are rectangles.
The other is a toroidal knot mosaic on the torus by identifying opposite boundary edges of 
a knot mosaic properly up to cyclic rotations, due to the topology of the torus.
The latter was introduced by Carlisle and Laufer~\cite{CL}.
Note that if we project a knot onto the torus instead of the plane,
we can usually lower the mosaic number that is another interesting invariant in knot mosaic theory.
This paper is inspired from Question 9 in~\cite{LK3} about the enumeration of knot mosaics,
and  Exercise 1 in~\cite{CL} for knot mosaics on the torus.

The authors have presented
several results in the research program related to the cardinality of knot mosaics
in the series of papers \cite{HLLO1, HLLO2, LHLO, Oh1, OHLL}.
We have developed a partition matrix argument to count knot mosaics of small size~\cite{HLLO2}.
This argument was later generalized to give an algorithm producing 
the exact enumeration of knot mosaics of any sizes,
which uses a recursion formula of so-called state matrices~\cite{OHLL}.
Also we refer~\cite{OhV1} for another application of this algorithm 
to the enumeration of independent vertex sets in grid graphs.

We apply the algorithm used in~\cite{OHLL} to count 
all period knot mosaics and toroidal knot mosaics.
The main difference from the classical knot mosaic theory is 
that one must ensure that the boundary tiles match appropriately.
We first count period knot mosaics in Section~\ref{sec:period}, 
and then count toroidal knot mosaics by considering equivalence classes of them under the cyclic rotations 
in both vertical and horizontal directions in Section~\ref{sec:toroidal}.
We also present the asymptotic behavior of the growth rates of their cardinality.

\section{Terminology and theorems}

Throughout this paper, the term `knot' means either a knot or a link.
Eleven symbols $T_0 \sim T_{10}$ illustrated in Figure~\ref{fig:tiles} are called {\em mosaic tiles\/}.

\begin{figure}[h]
\includegraphics{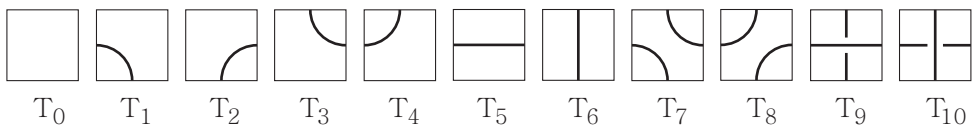}
\caption{Eleven mosaic tiles}
\label{fig:tiles}
\end{figure}

For positive integers $m$ and $n$,
an {\em $(m,n)$--mosaic\/} is an $m \! \times \! n$ matrix $M=(M_{ij})$ of mosaic tiles.
We denote the set of all $(m,n)$--mosaics by $\mathbb{M}^{(m,n)}$.
Note that $\mathbb{M}^{(m,n)}$ has $11^{mn}$ elements.
A {\em connection point\/} of a mosaic tile is defined as the midpoint of a mosaic tile edge
that is also the endpoint of a curve drawn on the tile.
Two tiles in a mosaic are called {\em contiguous\/} if they lie immediately next to each other
in either the same row or the same column.
We also say that two tiles are {\em $\partial$-contiguous\/}
if they lie on the opposite ends in either the same row or the same column.
A mosaic is called {\em suitably connected\/} if any pair of contiguous mosaic tiles
have or do not have connection points simultaneously on their common edge.
We also say that a mosaic is {\em suitably $\partial$-connected\/}
if any pair of $\partial$-contiguous mosaic tiles have or do not have
connection points simultaneously on their edges at the boundary of the mosaic.

\subsection{Knot mosaics}   \hspace{1cm}

We follow the original definition of a knot mosaic in~\cite{LK3}.

\begin{definition}
A knot $(m,n)$--mosaic is a suitably connected $(m,n)$--mosaic whose boundary edges
do not have connection points.
$\mathbb{K}^{(m,n)}$ denotes the subset of $\mathbb{M}^{(m,n)}$ of all knot $(m,n)$--mosaics.
$D^{(m,n)}$ denotes the cardinality of $\mathbb{K}^{(m,n)}$.
\end{definition}

Each knot $(m,n)$--mosaic represents a specific knot.
Two mosaics illustrated in Figure~\ref{fig:mosaic} represent 
a non-knot $(4,3)$--mosaic and the trefoil knot $(4,4)$--mosaic.

\begin{figure}[ht]
\includegraphics{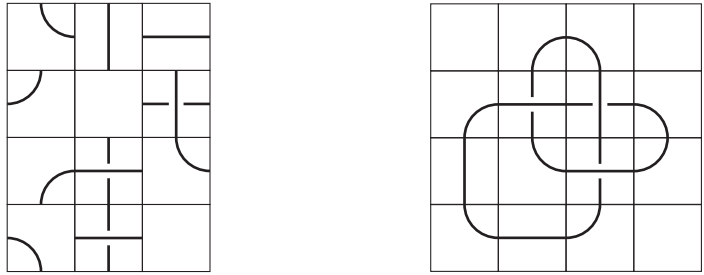}
\caption{Examples of mosaics}
\label{fig:mosaic}
\end{figure}

Several results about $D^{(m,n)}$ were founded by the authors
such as its lower and upper bounds in~\cite{HLLO1} 
and the precise values of $D^{(m,n)}$ for $m,n=4,5,6$ in~\cite{HLLO2}.
Recently the authors also constructed an algorithm
producing the precise value of $D^{(m,n)}$ in general
that uses a recursive matrix-relation that turns out to be remarkably efficient to count knot mosaics.
$\| N \|$ denotes the sum of all entries of a matrix $N$.

\begin{theorem} {\textup{(Oh-Hong-Lee-Lee~\cite{OHLL})}} \label{thm:knot}
For integers $m,n \geq 2$,
$$ D^{(m,n)} = 2 \, \| (X_{m-2}+O_{m-2})^{n-2} \|, $$
where $X_{m-2}$ and $O_{m-2}$ are $2^{m-2} \! \times \! 2^{m-2}$ matrices defined by
$$ X_{k+1} =
\begin{bmatrix} X_k & O_k \\ O_k & X_k  \end{bmatrix}
\ \mbox{and } \
O_{k+1} =
\begin{bmatrix} O_k & X_k \\ X_k & 4  O_k  \end{bmatrix} $$
for $k=0, 1, \dots, m-3$, starting with $1 \! \times \! 1$ matrices
$X_0 =  O_0 = \begin{bmatrix} 1 \end{bmatrix}$.
\end{theorem}

The algorithm used in Theorem~\ref{thm:knot} will be applied 
in the enumeration of the following period and toroidal knot mosaics.

\subsection{Period knot mosaics}   \hspace{1cm}

We now consider knot mosaics covering the whole plane in periodic patterns,
especially whose periodic patches are rectangles.
Thus any pair of $\partial$-contiguous mosaic tiles of each patch have or do not have
connection points simultaneously on their edges at the boundary.

\begin{definition}
A period knot $(m,n)$--mosaic is a suitably connected and suitably $\partial$-connected $(m,n)$--mosaic.
$\mathbb{K}_{P}^{(m,n)}$ denotes the set of all period knot $(m,n)$--mosaics.
$D_{P}^{(m,n)}$ denotes the cardinality of $\mathbb{K}_{P}^{(m,n)}$.
\end{definition}

\begin{figure}[ht]
\includegraphics{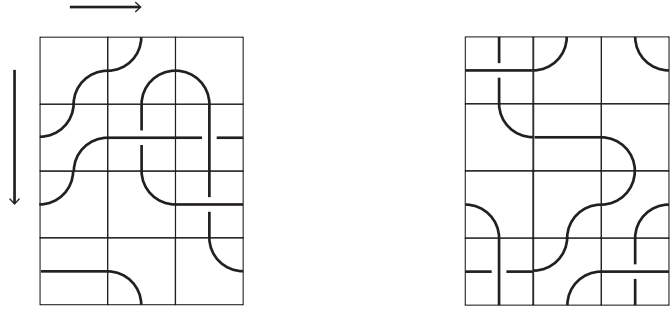}
\caption{Two different period knot mosaics, but the same toroidal knot mosaics}
\label{fig:pertor}
\end{figure}

In Figure~\ref{fig:pertor}, we present two different period knot $(4,3)$--mosaics.
The first main theorem is about an algorithm producing the precise value of $D_{P}^{(m,n)}$.
This theorem will be proved in Section~\ref{sec:period}.

\begin{theorem} \label{thm:period}
For positive integers $m$ and $n$,
$$ D_{P}^{(m,n)} = {\rm tr}(X_m^+ + O_m^+)^n, $$
where $X_m^+$ and $O_m^+$ are $2^m \! \times \! 2^m$ matrices defined by
$$ X_{k+1}^+ = \begin{bmatrix} X_k^+ & O_k^- \\ O_k^- & X_k^+ \end{bmatrix}, \ \  \ \ \
   X_{k+1}^- = \begin{bmatrix} X_k^- & O_k^+ \\ O_k^+ & X_k^- \end{bmatrix}, $$
$$ O_{k+1}^+ = \begin{bmatrix} O_k^+ & X_k^- \\ X_k^- & 4 \, O_k^+ \end{bmatrix} \ \ \mbox{and} \ \ \
   O_{k+1}^- = \begin{bmatrix} O_k^- & X_k^+ \\ X_k^+ & 4 \, O_k^- \end{bmatrix} $$
for $k=0, 1, \dots, m-1$, with
$ X_0^+ = O_0^+ = \begin{bmatrix} 1 \end{bmatrix} \mbox{ and } 
X_0^- = O_0^- = \begin{bmatrix} 0 \end{bmatrix}$.
\end{theorem}

Owing to Theorem~\ref{thm:period}, we get Table~\ref{tab1} of the precise values of $D_{P}^{(n,n)}$
and approximated values of $(D_{P}^{(n,n)})^{\frac{1}{n^2}}$.
$D_{P}^{(n,n)}$ seems to grow in a quadratic exponential rate and
this observation that $(D_{P}^{(n,n)})^{\frac{1}{n^2}}$ steadily decreases is 
of considerable significance\footnote{
{\bf Asymptotic behavior of the growth rate of $D_{P}^{(n,n)}$.} \
Consider a sequence of new matrices $\underline{O}_{k}^+$ satisfying the recurrence relation 
$\underline{O}_{k+1}^+ = \begin{bmatrix} \mathbb{O} & 
\mathbb{O} \\ \mathbb{O} & 4 \, \underline{O}_k^+ \end{bmatrix}$
starting with $\underline{O}_{0}^+ = \begin{bmatrix} 1 \end{bmatrix}$
where $\mathbb{O}$ is the square zero-matrix with an appropriate size.
Since $\underline{O}_{n}^+ \leq O_n^+ \leq X_n^+ + O_n^+$ (compare entrywise) and
$({\rm tr}(\underline{O}_{n}^+)^n)^{\ \frac{1}{n^2}}=4$,
$(D_{P}^{(n,n)})^{\frac{1}{n^2}}$ is always greater than or equal to 4.
For knot mosaics, it is shown~\cite{Oh1} that the limit $\lim_{n \rightarrow \infty} (D^{(n,n)})^{\ \frac{1}{n^2}}$ exists
and lies between $4$ and $\frac{5+ \sqrt{13}}{2} \ (\approx 4.303)$.
The existence of this growth constant relies on the two-variable version of Fekete's lemma and
the supermultiplicative property of $D^{(m,n)}$ in both indices so that
$D^{(m_1+m_2,n)} \geq D^{(m_1,n)} \cdot D^{(m_2,n)}$ and 
$D^{(m,n_1+ n_2)} \geq D^{(m,n_1)} \cdot D^{(m,n_2)}$.
But, for the toroidal case, $D_{P}^{(m,n)}$ looks like satisfying the submultiplicative property
by considering some numerical data.
If it is true, 
then similarly the limit $\lim_{n \rightarrow \infty} (D_{P}^{(n,n)})^{\ \frac{1}{n^2}}$ exists
and must be $\inf_{n \in \, \mathbb{N}} (D_{P}^{(n,n)})^{\ \frac{1}{n^2}}$,
and so it lies between 4 and 4.018454.
}.

\begin{table}[h]
{\footnotesize
\begin{tabular}{ccclccc}      \hline \hline
& $n$ && $D_{P}^{(n,n)}$ && $(D_{P}^{(n,n)})^{\frac{1}{n^2}}$ &  \\    \hline
& 1 && 7 && 7.000000 & \\
& 2 && 359 && 4.352849 & \\
& 3 && 316249 && 4.084269 & \\
& 4 && 4934695175 && 4.034863 & \\
& 5 && 1300161356831107  && 4.023091 & \\
& 6 && 5644698772550125092864 && 4.019872 & \\
& 7 && 399312236302057320966334185472 && 4.018911 & \\
& 8 && 457964061535512648565738757533162536960 && 4.018607 & \\
& 9 && 8496319497954601079390773421978474609756411527168 && 4.018506 & \\
& 10 && 2.54732361646079118531479661606646273328057$\cdots$e+60 && 4.018471 & \\
& 11 && 1.23368125451013250340475002575259970410360$\cdots$e+73 && 4.018459 & \\
& 12 && 9.64949082814445741693576869741862642790187$\cdots$e+86 && 4.018455 & \\
& 13 && 1.21885463463383945911667124257509803352769$\cdots$e+102 && 4.018453 & \\  
\hline \hline
\end{tabular}
}
\vspace{3mm}
\caption{List of $D_{P}^{(n,n)}$ and approximated $(D_{P}^{(n,n)})^{\frac{1}{n^2}}$}
\label{tab1}
\end{table}

\subsection{Toroidal knot mosaics}   \hspace{1cm}

We also consider knot mosaics on the torus rather than in the plane.
These mosaics can be obtained by identifying opposite boundaries of period knot mosaics.
Due to the topology of the torus, 
they must be treated as equivalence classes under the cyclic rotations meridionally and logitudally.
We say two period knot mosaics are {\em equivalent\/}
if one can be obtained from the other by a finite sequence of cyclic rotations of rows and columns.

\begin{definition}
A toroidal knot $(m,n)$--mosaic is an equivalence class of
suitably connected and suitably $\partial$-connected $(m,n)$--mosaics.
$\mathbb{K}_T^{(m,n)}$ denotes the set of all toroidal knot $(m,n)$--mosaics.
$D_T^{(m,n)}$ denotes the cardinality of $\mathbb{K}_T^{(m,n)}$.
\end{definition}

Two examples in Figure~\ref{fig:pertor} represent the same toroidal knot $(4,3)$--mosaic.
To get one from the other, we take cyclic rotations by two rows and one column in the directions of the arrows.
$\mathbb{K}_T^{(m,n)}$ can be obtained from $\mathbb{K}_{P}^{(m,n)}$ by a relevant quotient map.
The next two theorems are about algorithms producing the precise value of $D_T^{(m,n)}$.
These theorems will be proved in Section~\ref{sec:toroidal}.
Here $p | m$ means that $m$ is divisible by $p$.
Define, for a $2^p \! \times \! 2^p$ matrix $A$, 
$${\rm tr}^{(k)}(A) = \sum_{i=0}^{2^p-1} (i \! + \! 1,\alpha(i) \! + \! 1)\mbox{-entry of } A,$$
where $\alpha(i) = 2^k i \ (\mbox{mod } 2^p \! - \! 1)$. 
Especially ${\rm tr}^{(0)}(A) = {\rm tr}(A)$.

\begin{theorem} \label{thm:toroidalmn}
For positive co-prime integers $m$ and $n$,
$$D_T^{(m,n)} = \sum_{p | m, \, q | n} \frac{1}{pq} d_{p,q},$$
where positive integers $d_{p,q}$ are recursively defined by
$$d_{p,q} = D_{P}^{(p,q)} - \sum_{r | p, \, s | q, \, rs \neq pq} d_{r,s}.$$
\end{theorem}

\begin{theorem} \label{thm:toroidalpp}
$D_T^{(2,2)} = 110$. And for a prime integer $p \geq 3$,
$$D_T^{(p,p)} = \frac{1}{p^2} d_{p^2} + \frac{2}{p} \ \sum_{k=0}^{\frac{p-1}{2}} d_{p_{{(k,1)}}} + 7,$$
where positive integers $d_{p_{{(k,1)}}}$ and $d_{p^2}$ are defined by
$$d_{p_{{(k,1)}}} = {\rm tr}^{(k)}(X_p^{+} + O_p^{+}) -7 \mbox{ and } \
d_{p^2} = D_{P}^{(p,p)} - 2 \sum_{k=0}^{\frac{p-1}{2}} d_{p_{{(k,1)}}} - 7.$$
\end{theorem}

Due to these two theorems, we get Table~\ref{tab2} of the precise values of $D_T^{(m,n)}$.
To handle the cases $(m,n)=(2,4)$ and $(4,4)$ in the table, 
we can apply the argument used in the proof of Theorem~\ref{thm:toroidalpp} with slight adjustment.
Toroidal knot mosaics have the same asymptotic behavior\footnote{
{\bf Asymptotic behavior of the growth rate of $D_{T}^{(n,n)}$.} \
We easily deduce that 
$\frac{1}{mn} D_{P}^{(m,n)} \leq D_{T}^{(m,n)} \leq D_{P}^{(m,n)}$
for any positive integers $m$ and $n$
because at most $mn$ different period knot mosaics can be produced
from a toroidal knot mosaic by cyclic rotations.  
This implies that $(D_{T}^{(n,n)})^{\ \frac{1}{n^2}}$ and $(D_{P}^{(n,n)})^{\ \frac{1}{n^2}}$
have the same asymptotic behavior.
} as period knot mosaics.
Remark that, for toroidal knot $(2,2)$--mosaics, Carlisle and Laufer~\cite{CL} found the catalog of all such mosaics,
but there are 22 missing mosaics\footnote{
For toroidal knot $(2,2)$--mosaics, $D_T^{(2,2)}=110$.
In~\cite{CL}, the catalog of all such mosaics is found,
but there are 98 mosaics and among them 10 mosaics are counted twice
($K_{12}=K_{13}$, $K_{14}=K_{15}$, $K_{28}=K_{29}$, $K_{30}=K_{31}$, $K_{65}=K_{78}$, 
$K_{66}=K_{77}$, $K_{67}=K_{79}$, $K_{68}=K_{81}$, $K_{70}=K_{82}$, $K_{99}=K_{104}$).
The 22 missing mosaics are drawn in Figure~\ref{fig:missing}
}.

\begin{table}[h]
\begin{tabular}{cccccc}      \hline \hline
\ $D_T^{(m,n)}$ \ & \ $n=1$ \ & \ $n=2$ \ & \ $n=3$ \ & $n=4$  & $n=5$ \\   \hline
$m=1$ & 7 & 18 & 49 & 171 & 637 \\
$m=2$ &   & 110 & 954 & 11591 & 155310 \\
$m=3$ &   &  & 35237  & 1662837 & 86538181 \\
$m=4$ & & & & 308435024 & 63440607699 \\
$m=5$ & & & &  & 52006454275147 \\  \hline \hline
\end{tabular}
\vspace{3mm}
\caption{List of $D_T^{(m,n)}$}
\label{tab2}
\end{table}

\begin{figure}[ht]
\includegraphics{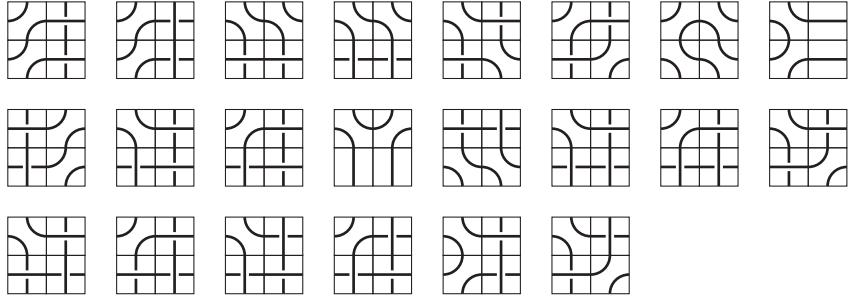}
\caption{22 missing toroidal knot $(2,2)$--mosaics}
\label{fig:missing}
\end{figure}

\section{State matrices}

In this section, we review the notion in~\cite{OHLL}.
Let $m$ and $n$ be positive integers.
$\mathbb{S}^{(m,n)}$ denotes the set of all suitably connected $(m,n)$--mosaics.
So each mosaic possibly has connection points on its boundary edges.
For example, a suitably connected (3,5)--mosaic is depicted in Figure~\ref{fig:suitable}.

\begin{figure}[ht]
\includegraphics{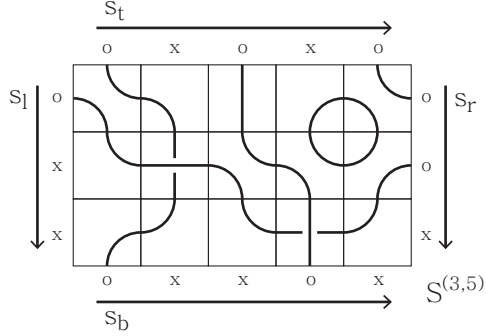}
\caption{Suitably connected (3,5)--mosaic $S^{(3,5)}$}
\label{fig:suitable}
\end{figure}

For simplicity of exposition, a mosaic tile is called $l$--, $r$--, $t$-- and $b$--{\em cp\/}
if it has a connection point on its left, right, top and bottom boundaries, respectively.
We use two or more letters such as $lt$--cp for the case of both $l$--cp and $t$--cp.
The sign $\hat{}$ \/ is for negation so that, for example,
$\hat{l}$--cp means not $l$--cp, and
$\hat{l} \hat{t}$--cp means both $\hat{l}$--cp and $\hat{t}$--cp. \\

\noindent {\bf Choice rule.\/}
Each $M_{ij}$ in a mosaic has four choices of mosaic tiles
as $T_7$, $T_8$, $T_9$ and $T_{10}$ if it has four connection points (i.e., $lrtb$--cp).
It has unique choice if it has no or two connection points.
It does not have odd number of connection points. \\

For a suitably connected $(m,n)$--mosaic $S^{(m,n)} =  (M_{ij})$ where $i=1,\dots,m$ and $j=1,\dots,n$,
an {\em $l$--state\/} of $S^{(m,n)}$ indicates the presence of connection points
on its left boundary edge,
and we denote that $s_l(S^{(m,n)}) = s_l(M_{11}) s_l(M_{21}) \cdots s_l(M_{m1})$
where $s_l(M_{ij})$ denotes ``x'' if $M_{ij}$ is $\hat{l}$--cp and ``o'' if $M_{ij}$ is $l$--cp.
We similarly define $r$--, $t$-- and $b$--states of $S^{(m,n)}$
that indicate the presence of connection points on its right, top and bottom boundary edges, respectively.
For example, the suitably connected $(3,5)$--mosaic $S^{(3,5)}$ drawn in Figure~\ref{fig:suitable} has
$s_l(S^{(3,5)}) =$ oxx, $s_r(S^{(3,5)}) =$ oox, $s_t(S^{(3,5)}) =$ oxoxo and $s_b(S^{(3,5)}) =$ oxxox.
Note that $\mathbb{S}^{(m,n)}$ has possibly $2^m$ kinds of $l$--states and also $2^m$ kinds of $r$--states.
Now we arrange the elements of the set of all states in the reverse lexicographical order such as
xxx, oxx, xox, oox, xxo, oxo, xoo and ooo for $m=3$.

A {\em state matrix\/} for $\mathbb{S}^{(m,n)}$  is a $2^m \! \times \! 2^m$ matrix $N^{(m,n)} = (x_{ij})$
where each entry $x_{ij}$ is the number of all suitably connected $(m,n)$--mosaics
that have the $i$-th $l$--state and the $j$-th $r$--state in the set of $2^m$ states of the order arranged above.
For $n=1$, we split the state matrix $N^{(m,1)}$ into four $2^m \! \times \! 2^m$ matrices,
namely $X_m^+$, $X_m^-$, $O_m^+$ and $O_m^-$ as follows:
each $(i,j)$-entry of $X_m^+$ ($X_m^-$, $O_m^+$ or $O_m^-$) indicates
the number of all suitably connected $(m,1)$--mosaics
that have the $i$-th $l$--state and the $j$-th $r$--state,
and additionally whose ($b$--state, $t$--state) is (x, x)
((x, o), (o, o) or (o, x), respectively).
Note that the letters $X$ and $O$ depend on their $b$--states and
we use the sign $\mbox{}^+$ (or $\mbox{}^-$)
when they have the same (different, respectively) $b$--state and $t$--state.
Note that $N^{(m,1)} = X_m^+ + X_m^- + O_m^+ + O_m^-$.

\section{Period knot mosaics} \label{sec:period}

In this section, we prove Theorem~\ref{thm:period}.
Note that any period knot mosaic has the same $t$--state and $b$--state,
and the same $l$--state and $r$--state because of the suitable $\partial$-connectedness.
First, consider the subset $\mathbb{S}^{(m,n)+}$ of $\mathbb{S}^{(m,n)}$
consisting of all suitably connected $(m,n)$--mosaics
each of which has the same $t$--state and $b$--state.
$N^{(m,n)+}$ denotes a state matrix for $\mathbb{S}^{(m,n)+}$.
Obviously $N^{(m,1)+} = X_m^+ + O_m^+$.
For a $2^{k+1} \! \times \! 2^{k+1}$ matrix $N = (x_{ij})$,
the 11-quadrant (similarly 12-, 21- or 22-quadrant) of $N$ denotes
the $2^k \! \times \! 2^k$ submatrix $(x_{ij})$ where $1 \leq i, j \leq 2^k$
($1 \leq i \leq 2^k$ and $2^k+1 \leq j \leq 2^{k+1}$,
$2^k+1 \leq i \leq 2^{k+1}$ and $1 \leq j \leq 2^k$,
or $2^k+1 \leq i, j \leq 2^{k+1}$, respectively).

\subsection{State matrices for $(1,1)$--mosaics}   \hspace{1cm}

We construct the following matrices directly from Figure~\ref{fig:state11}.
The entries 0, 1 and 4 are determined by Choice rule.
As an example, $(2,2)$-entry of $O_1^+$ is 4 
because the associated mosaic tiles must have four connection points.
$$X_1^+ = \begin{bmatrix} 1 & 0 \\ 0 & 1 \end{bmatrix}, \ \
X_1^- = \begin{bmatrix} 0 & 1 \\ 1 & 0 \end{bmatrix}, \ \
O_1^+ = \begin{bmatrix} 1 & 0 \\ 0 & 4 \end{bmatrix}, \ \
O_1^- = \begin{bmatrix} 0 & 1 \\ 1 & 0 \end{bmatrix} $$
$$\mbox{and} \ \
N^{(1,1)+} = X_1^+ + O_1^+ = \begin{bmatrix} 2 & 0 \\ 0 & 5 \end{bmatrix}.$$
Indeed the mosaics counted in the diagonal entries of $N^{(1,1)+}$ are
the only seven period knot $(1,1)$--mosaics, $T_0$ and $T_5 \sim T_{10}$,
among eleven $(1,1)$--mosaics as shown in the figure.

\begin{figure}[ht]
\includegraphics{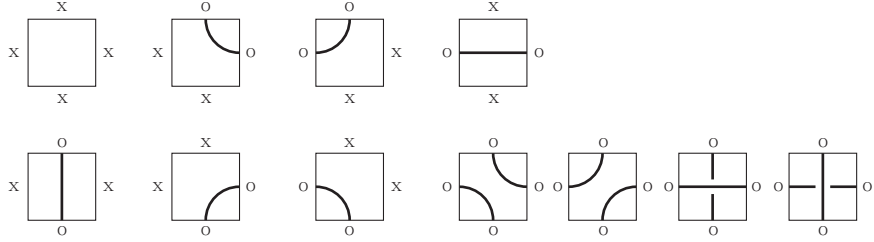}
\caption{States of $(1,1)$--mosaics}
\label{fig:state11}
\end{figure}

\subsection{State matrix $N^{(m,1)+}$ for $\mathbb{S}^{(m,1)+}$}   \hspace{1cm}

We follow the inductive proof of Proposition 2 in~\cite{OHLL} with modification.
The matrices $X_1^+$, $X_1^-$, $O_1^+$ and $O_1^-$ are already known.
Assume that the matrices $X_k^+$, $X_k^-$, $O_k^+$ and $O_k^-$ satisfy the statement.
We will find $O_{k+1}^+$, and the readers can easily apply this method
to find the rest $X_{k+1}^+$, $X_{k+1}^-$ and $O_{k+1}^-$.
All entries of $O_{k+1}^+$ count the suitably connected $(k+1,1)$--mosaics $S^{(k+1,1)} = (M_{i,1})$,
$i=1, \dots, k+1$, in $\mathbb{S}^{(k+1,1)}$ whose ($b$--state, $t$--state) is (o,~o).

If the bottom mosaic tile $M_{k+1,1}$ is $\hat{l} \hat{r}$--cp,
then $S^{(k+1,1)}$ should be counted in an entry of the 11-quadrant of $O_{k+1}^+$
because of the reverse lexicographical order of $2^{k+1}$ states.
In this case, $M_{k+1,1}$ must be the mosaic tile $T_6$.
Let $S^{(k,1)}$ be the associated suitably connected $(k,1)$--mosaic
obtained from $S^{(k+1,1)}$ by deleting $M_{k+1,1}$.
Then ($b$--state, $t$--state) of $S^{(k,1)}$ is also (o, o),
and so the associated state matrix for all possible $S^{(k,1)}$ is $O_{k}^+$.
Indeed $l$--cp (or $r$--cp) is related to $2^k+1 \leq i \leq 2^{k+1}$
(or $2^k+1 \leq j \leq 2^{k+1}$, respectively).
Figure~\ref{fig:findSM} and Table~\ref{tab3} explain all four cases according to
the $l$-- and $r$--states of $M_{k+1,1}$.
Notice that only when $M_{k+1,1}$ is $lrb$--cp,
it has four choices of mosaic tiles $T_7$, $T_8$, $T_9$ and $T_{10}$.
Thus the associated submatrix must be $4 O_k^+$ instead of $O_k^+$.
Now, we get the recurrence relation as
$$ X_{k+1}^+ = \begin{bmatrix} X_k^+ & O_k^- \\ O_k^- & X_k^+ \end{bmatrix}, \ \  \ \ \
   X_{k+1}^- = \begin{bmatrix} X_k^- & O_k^+ \\ O_k^+ & X_k^- \end{bmatrix}, $$
$$ O_{k+1}^+ = \begin{bmatrix} O_k^+ & X_k^- \\ X_k^- & 4 \, O_k^+ \end{bmatrix} \ \ \mbox{and} \ \ \
   O_{k+1}^- = \begin{bmatrix} O_k^- & X_k^+ \\ X_k^+ & 4 \, O_k^- \end{bmatrix}. $$

\begin{figure}[ht]
\includegraphics{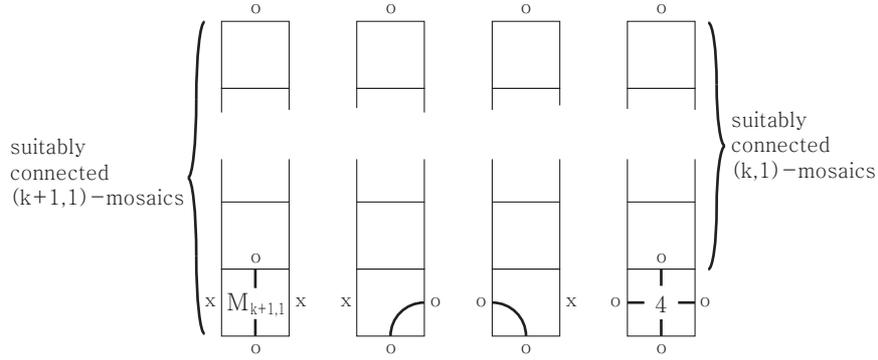}
\caption{Finding the state matrix $O_{k+1}^+$}
\label{fig:findSM}
\end{figure}

\begin{table}[h]
\begin{tabular}{cccc}      \hline  \hline
 & {\em Quadrants\/} & {\em Associated $M_{k+1,1}$\/} & \ {\em Submatrices\/} \ \\    \hline
\multirow{4}{8mm}{$O_{k+1}^+$}
 & \ 11-quadrant ($\hat{l} \hat{r} b$--cp) \ & $T_6$ & $O_k^+$ \\
 & 12-quadrant ($\hat{l} r b$--cp) & $T_2$ & $X_k^-$ \\
 & 21-quadrant ($l \hat{r} b$--cp) & $T_1$ & $X_k^-$ \\
 & 22-quadrant ($l r b$--cp) & \ $T_7$, $T_8$, $T_9$, $T_{10}$ \ & $4 O_k^+$ \\  \hline  \hline
\end{tabular}
\vspace{3mm}
\caption{Four quadrants of $O_{k+1}^+$}
\label{tab3}
\end{table}

\subsection{State matrix $N^{(m,n)+}$ for $\mathbb{S}^{(m,n)+}$}   \hspace{1cm}

We directly follow the proof of Proposition 3 in~\cite{OHLL}.
Using $\mathbb{S}^+$ and $N^+$ instead of $\mathbb{S}$ and $N$, respectively,
is the only difference.
We reprove it for self-containedness of the paper.

\begin{lemma} \label{lem:npq}
$N^{(m,n)+} = (N^{(m,1)+})^n = (X_m^+ + O_m^+)^n.$
\end{lemma}

\begin{proof}
Use the induction on $n$.
Assume that $N^{(m,k)+} = (N^{(m,1)+})^k$.
For a mosaic $S^{(m,k+1)}$ in $\mathbb{S}^{(m,k+1)+}$,
split the mosaic into two suitably connected $(m,k)$-- and $(m,1)$--mosaics
$S^{(m,k)}$ and $S^{(m,1)}$ by taking the left $k$ columns and the rightmost column, respectively.
Then $r$--state of $S^{(m,k)}$ is the same as $l$--state of $S^{(m,1)}$ because of suitable connectedness.
Remark that $N^{(m,k+1)+} = (x^{(k+1)}_{ij})$ is the state matrix for $\mathbb{S}^{(m,k+1)+}$
where each entry $x^{(k+1)}_{ij}$ counts the number of all suitably connected $(m,k \! + \! 1)$--mosaics
each of which has the same $t$--state and $b$--state,
and the $i$-th $l$--state and the $j$-th $r$--state in the set of $2^m$ states.
Two state matrices $N^{(m,k)+} = (x^{(k)}_{ij})$
and $N^{(m,1)+} = (x^{(1)}_{ij})$ are defined similarly.
Among these suitably connected $(m,k \! + \! 1)$--mosaics counted in each entry $x^{(k+1)}_{ij}$,
the number of all mosaics whose $r$--state of the $k$-th column
(or equally $l$--state of the $(k \! + \! 1)$-th column) is the $s$-th state
is the product of $x^{(k)}_{is}$ and $x^{(1)}_{sj}$.
Since all $2^m$ states can be appeared as a state of connection points
where $S^{(m,k)}$ and $S^{(m,1)}$ meet,
we get
$$ x^{(k+1)}_{ij} = \sum^{2^m}_{s=1} x^{(k)}_{is} \cdot x^{(1)}_{sj}.$$
This implies that
$ N^{(m,k+1)+} = N^{(m,k)+} \cdot N^{(m,1)+} = (N^{(m,1)+})^{k+1}$.
\end{proof}

\begin{proof}[Proof of Theorem~\ref{thm:period}.]
For positive integers $m$ and $n$,
recall that $N^{(m,n)+}$ is the state matrix for $\mathbb{S}^{(m,n)+}$ that is the set of
all suitably connected $(m,n)$--mosaics each of which has the same $t$--state and $b$--state.
To be a period knot $(m,n)$--mosaic 
(i.e., it additionally satisfies the suitable $\partial$-connectedness),
it must also have the same $l$--state and $r$--state.
The only mosaics counted in the diagonal entries of $N^{(m,n)+}$ have this property.
Thus,
$$ D_{P}^{(m,n)} = {\rm tr}(N^{(m,n)+}) = {\rm tr}(X_m^+ + O_m^+)^n.$$
Note that for the initial condition of the recursion formula,
we may start with the seed matrices $X_0^+ = O_0^+ = \begin{bmatrix} 1 \end{bmatrix}$
and $X_0^- = O_0^- = \begin{bmatrix} 0 \end{bmatrix}$,
instead of $X_1^+$, $O_1^+$, $X_1^-$ and $O_1^-$.
\end{proof}

\section{Toroidal knot mosaics} \label{sec:toroidal}

In this section, we prove Theorems~\ref{thm:toroidalmn} and \ref{thm:toroidalpp}.
Recall that two period knot mosaics $M$ and $M'$ are equivalent
if one can be obtained from the other by a finite sequence of cyclic rotations.
Let $M=(M_{i,j})$ be a period knot $(m,n)$--mosaic.
$[M]$ denotes a toroidal knot mosaic that is an equivalence class of $M$.
Let $f$ be the quotient map from $\mathbb{K}_{P}^{(m,n)}$ to $\mathbb{K}_T^{(m,n)}$
defined by $f(M) = [M]$.
We define $t_{x,y}(M)=(M'_{i,j})$ where $M'_{i,j} = M_{i-x,j-y}$ for all $i,j$,
performing cyclic rotations by $x$ rows and $y$ columns.
We use sets of indices $\{ 1,2, \dots, m \}$ and $\{ 1,2, \dots, n \}$ as complete residue systems
modulo $m$ and $n$, respectively.
In $\mathbb{K}_T^{(m,n)}$, $[M] = [t_{x,y}(M)]$ for all $x$ and $y$.

$M$ is called a {\em $(p,q)$--f.period\/} knot $(m,n)$--mosaic
(or, a period knot $(m,n)$--mosaic with a fundamental period $(p,q)$)
if $p$ and $q$ are smallest positive integers
such that $M = t_{p,0}(M) = t_{0,q}(M)$ as the top figure in Figure~\ref{fig:periodicexam}.
Thus $M = t_{ap,bq}(M)$ for any integers $a$ and $b$.
Note that $M = t_{p,q}(M)$ is not the sufficient condition of being 
a $(p,q)$--f.period  knot $(m,n)$--mosaic.
Let $d_{p,q}$ be the total number of all $(p,q)$--f.period knot $(m,n)$--mosaics.
Note that $d_{1,1} = 7$.

\begin{figure}[ht]
\includegraphics{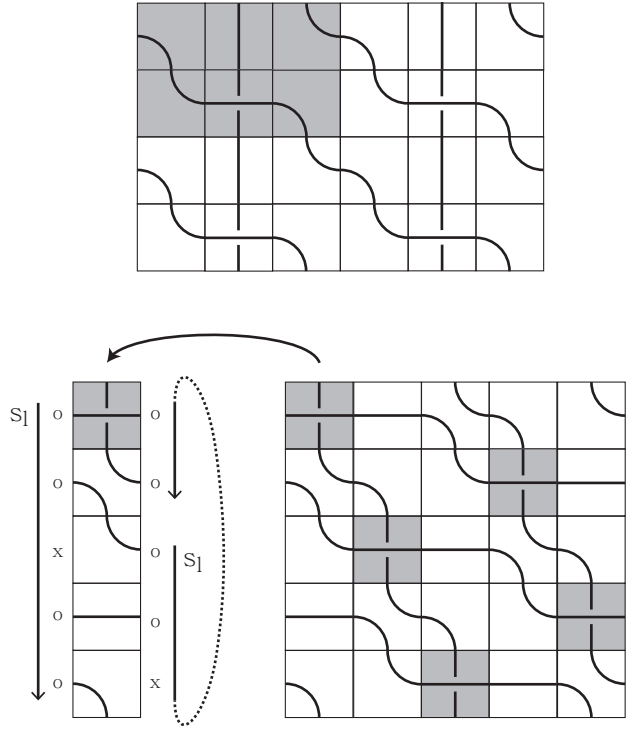}
\caption{A $(2,3)$--f.period knot $(4,6)$--mosaic (top) and 
a $5_{(2,1)}$--f.period knot $(5,5)$--mosaic (bottom)}
\label{fig:periodicexam}
\end{figure}

\begin{proof}[Proof of Theorem~\ref{thm:toroidalmn}.]
Suppose that $m$ and $n$ are positive co-prime integers.
If $M$ is a $(p,q)$--f.period knot $(m,n)$--mosaic,
then $m$ and $n$ must be divisible by $p$ and $q$, respectively.
Otherwise, assume that $m$ is not divisible by $p$.
There are two positive integers $a,b$ such that $ap = m+b$ where $1 \leq b < p$.
Then $M = t_{ap,0}(M) = t_{b,0}(M)$, and so $M$ is a $(b,q)$--f.period knot $(m,n)$--mosaic
for $b$ which is less than $p$, a contradiction.
Therefore $\mathbb{K}_{P}^{(m,n)}$ is a disjoint union of sets of
all $(p,q)$--f.period knot $(m,n)$--mosaics
for all possible pairs of factors $p$ and $q$ of $m$ and $n$, respectively.
Thus $D_{P}^{(m,n)} = \sum_{p | m, \, q | n} d_{p,q}$.

Indeed a $(p,q)$--f.period knot $(m,n)$--mosaic is merely
$\frac{m}{p} \times \frac{n}{q}$ copies of a period knot $(p,q)$--mosaic arrayed as a checkerboard.
This guarantees that for any $p | m$ and $r | p$, and $q | n$ and $s | q$,
the total number $d_{r,s}$ of all $(r,s)$--f.period knot $(m,n)$--mosaics
is the same as the total number of all $(r,s)$--f.period knot $(p,q)$--mosaics.
Thus we have a similar equation 
$D_{P}^{(p,q)} = \sum_{r | p, \, s | q} d_{r,s}$, or
$$d_{p,q} = D_{P}^{(p,q)} - \sum_{r | p, \, s | q, \, rs \neq pq} d_{r,s}.$$

Now consider the set $f^{-1}([M])$ of the pre-image of $[M]$ 
for a $(p,q)$--f.period knot $(m,n)$--mosaic $M$.
We will show that $f^{-1}([M])$ consists of exactly $pq$ totally different elements,
i.e., all cyclic rotations $t_{x,y}(M)$'s for $x=0,1, \dots, p \! - \! 1$ and  $y=0,1, \dots, q \! - \! 1$ are totally different.
Assume for contradiction that $t_{x',y'}(M) = t_{x'',y''}(M)$ for different pairs of integers,
or $t_{x,y}(M) = M$ where $x=x'-x''$ and $y=y'-y''$.
We may say that $x=1, \dots,$ or $p \! - \! 1$ and $y=1, \dots,$ or $q \! - \! 1$.
Let $g$ be the greatest common divisor of $x$ and $p$.
Then $M = t_{\frac{p}{g} x, \frac{p}{g} y}(M) = t_{0,\frac{py}{g}}(M)$
because $\frac{p}{g}$ is a positive integer and $\frac{p}{g} x$ is indeed divisible by $p$.
But $\frac{py}{g}$ can not be divisible by $q$ because $p$ and $q$ are co-prime
and $\frac{y}{g}$ is smaller than $q$.
Let $k$ be the non-zero remainder of $\frac{py}{g}$ divided by $q$.
Then $M = t_{0,\frac{py}{g}}(M) = t_{0,k}(M)$, and so $M$ is a $(p,k)$--f.period knot $(m,n)$--mosaic.

This proves that for any pairs of factors $p$ and $q$ of $m$ and $n$, respectively,
the total number of all toroidal knot $(m,n)$--mosaics with a fundamental period $(p,q)$ 
is $\frac{1}{pq} d_{p,q}$.
Therefore 
$$D_T^{(m,n)} = \sum_{p | m, \, q | n} \frac{1}{pq} d_{p,q}.$$
\end{proof}

Now consider a period knot $(p,p)$--mosaic $M$ in 
$\mathbb{K}_{P}^{(p,p)}$ for a prime integer~$p$.
$M$ is called a {\em $p_{(k,1)}$--f.period\/} knot $(p,p)$--mosaic for $k=0,1,\dots,p \! - \! 1$ 
if it is not a $(1,1)$--f.period knot $(p,p)$--mosaic and 
$M = t_{k,1}(M)$ as the bottom figure in Figure~\ref{fig:periodicexam}.
Especially a $p_{(0,1)}$--f.period knot $(p,p)$--mosaic means 
a $(p,1)$--f.period knot $(p,p)$--mosaic.
Let $d_{p_{{(k,1)}}}$ be the total number of all $p_{(k,1)}$--f.period knot $(p,p)$--mosaics
and $d_{p^2}$ be the rest period knot $(p,p)$--mosaics
which are not $(1,1)$--, $(1,p)$-- or $p_{(k,1)}$--f.period knot $(p,p)$--mosaics for all $k$.

\begin{proof}[Proof of Theorem~\ref{thm:toroidalpp}.]
Suppose that $p$ is a prime integer.
We easily know that $(1,1)$--, $(1,p)$-- and 
$p_{(k,1)}$--f.period knot $(p,p)$--mosaics for all $k$ are totally different.
Therefore $D_{P}^{(p,p)} = d_{p^2} + d_{1,p} + \sum_{k=0}^{p-1} d_{p_{{(k,1)}}} + d_{1,1}$.
Since $d_{1,p} = d_{p,1} = d_{p_{{(0,1)}}}$ and $d_{p_{{(k,1)}}} = d_{p_{{(p-k,1)}}}$ 
from the natural symmetries of the torus,
$$d_{p^2} = D_{P}^{(p,p)} -  2 \sum_{k=0}^{\frac{p-1}{2}} d_{p_{{(k,1)}}} - 7,$$
except when $p=2$,
$d_{2^2} = D_{P}^{(2,2)} -  2 d_{2_{{(0,1)}}} - d_{2_{{(1,1)}}} - 7$.
Recall that $d_{1,1} = 7$.

Now we count the number $d_{p_{{(k,1)}}}$ of $p_{(k,1)}$--f.period $(p,p)$--mosaics.
From the definition, we only need to count the number of suitably connected $(p,1)$--mosaics
$M = (M_{i,1})$ such that $s_t(M)=s_b(M)$ and $s_l(M_{i,1})=s_r(M_{i+k,1})$ for all $k$.
The latter means $g_k(s_l(M))=s_r(M)$
where $g_k$ shifts a word to the right by $k$ letters cyclicwise,
as for example $g_2(abcde)=deabc$.
Indeed, $g_k$ send an $(i \! + \! 1)$-th state among $2^p$ states to an $(\alpha(i) \! + \! 1)$-th state
where $\alpha(i) = 2^k i \ (\mbox{mod } 2^p \! - \! 1)$.
Therefore such $(p,1)$--mosaics are counted in $(i \! + \! 1, \alpha(i) \! + \! 1)$-entries of $X^+_m + O^+_m$
for $i=0,1, \dots, 2^p-1$.
Note that $d_{p_{{(k,1)}}}$ does not count seven $(1,1)$--f.period $(p,p)$--mosaics.
Therefore,
$$d_{p_{{(k,1)}}} = {\rm tr}^{(k)}(X_p^{+} + O_p^{+}) - 7.$$

Now consider the set $f^{-1}([M])$ of the pre-image of $[M]$.
If $M$ is a $p_{(k,1)}$--f.period $(p,p)$--mosaic, 
$f^{-1}([M])$ consists of exactly $p$ totally different elements
which are all cyclic rotations $t_{x,0}(M)$'s for $x=0,1, \dots, p-1$.
If $M$ is not $(1,1)$--, $(1,p)$-- or $p_{(k,1)}$--f.period $(p,p)$--mosaics for all $k$,
$f^{-1}([M])$ consists of exactly $p^2$ totally different elements.
Therefore,
$$D_T^{(p,p)} = \frac{1}{p^2} d_{p^2} + \frac{2}{p} \ \sum_{k=0}^{\frac{p-1}{2}} d_{p_{{(k,1)}}} + 7,$$
except when $p=2$,
$D_T^{(2,2)} = \frac{1}{2^2} d_{2^2} + \frac{1}{2} (2  d_{2_{{(0,1)}}} + d_{2_{{(1,1)}}}) + 7 = 110$.
\end{proof}

\end{document}